\documentclass{amsart}

\newtheorem{thm}{Theorem}[section]

\theoremstyle{definition}

\theoremstyle{remark}
\newtheorem{rem}[thm]{Remark}

\numberwithin{equation}{section}

\begin{document}
\title[Stability of a functional equation]
{Stability of the parametric fundamental equation
of information for nonpositive parameters}

\author{Eszter Gselmann}
\address{
Institute of Mathematics\\
University of Debrecen\\
P. O. Box: 12.\\
Debrecen\\
Hungary\\
H--4010}
\email{gselmann@math.klte.hu}

\author{Gyula Maksa}
\address{Institute of Mathematics\\
University of Debrecen\\
P. O. Box: 12.\\
Debrecen\\
Hungary\\
H--4010}
\email{maksa@math.klte.hu}

\begin{abstract}
In this note we prove that the parametric fundamental equation
of information is stable in the sense of Hyers and Ulam
provided that the parameter is nonpositive.
We also prove, as a corollary, that the system of equations that
defines the recursive and semi-symmetric information measures
depending on a nonpositive parameter is stable in a certain sense.
\end{abstract}

\thanks{This research has been supported by the Hungarian Scientific Research Fund
(OTKA) Grants NK 68040 and K 62316.}
\maketitle

\section{Introduction}

The basic problem in the stability theory of functional
equations is whether an approximate solution of a
functional equation or a system of functional equations
can be approximated by a solution of the equation or the
system of equations in question.

In this paper we prove that the parametric fundamental equation of information
\begin{equation}\label{Eq1.1}
f(x)+(1-x)^{\alpha}f\left(\frac{y}{1-x}\right)=
f(y)+(1-y)^{\alpha}f\left(\frac{x}{1-y}\right)
\end{equation}
is stable in the sense of Hyers and Ulam (see the expository papers
Forti \cite{For95}, Ger \cite{Ger94}, Moszner \cite{Mos04}),
provided that $\alpha$ is nonpositive.
Equation (\ref{Eq1.1}) arises in a natural way in
characterizing information measures based on the properties of
$\alpha$--recursivity  and semi--symmetry (see Acz\'{e}l--Dar\'{o}czy \cite{AD75}).
In the investigations (\ref{Eq1.1}) is supposed to hold on
\[
D=\left\{(x, y)\in \mathbb{R}^{2} \vert x, y \in [0,1[, x+y\leq 1\right\}
\]
with $f:[0,1]\rightarrow\mathbb{R}$ or only on the interior of $D$,
\[
D^{\circ}=\left\{(x, y)\in\mathbb{R}^{2}\vert x, y, x+y\in ]0,1[\right\},
\]
with $f:]0,1[\rightarrow\mathbb{R}$ (see Acz\'{e}l--Dar\'{o}czy \cite{AD75},
Acz\'{e}l \cite{Acz81}, \cite{Acz86} and their references).
In \cite{Mak08} we proved that (\ref{Eq1.1}) is stable on $D$,
moreover it is superstable (see Forti \cite{For95}) if $\alpha>0$
and $\alpha\neq 1$.
The question of the stability of (\ref{Eq1.1}) in the exceptional case
$\alpha=1$ on $D^{\circ}$ was raised by Sz\'{e}kelyhidi \cite{Szek91}.
The method we used in \cite{Mak08} can not be applied
neither in this case nor in the case $\alpha\leq 0$
neither on $D$ nor on $D^{\circ}$.
Finally, we should remark that the ideas we use in this
paper to prove the stability of (\ref{Eq1.1}) on $D^{\circ}$
as well as on $D$ do not work if $\alpha>0$.

\section{The main result}

First we prove the following.

\begin{thm}\label{Thm2.1}
Let $\alpha, \varepsilon \in\mathbb{R}$ be fixed, $\alpha\leq 0$ and
$\varepsilon \geq 0$.
Suppose that the function $f:]0,1[\rightarrow\mathbb{R}$ satisfies
the inequality
\begin{equation}\label{Eq2.1}
\left|f(x)+(1-x)^{\alpha}f\left(\frac{y}{1-x}\right)
-f(y)-(1-y)^{\alpha}f\left(\frac{x}{1-y}\right)\right| \leq \varepsilon
\end{equation}
for all $(x, y)\in D^{\circ}$.
Then, in case $\alpha<0$ there exist $a, b\in\mathbb{R}$ such that
\begin{equation}\label{Eq2.2}
\left|f(x)-\left[ax^{\alpha}+b(1-x)^{\alpha}-b\right]\right| \leq 15 \varepsilon,
\quad \left(x\in ]0,1[\right)
\end{equation}
furthermore, in case $\alpha=0$, there exists a logarithmic function
$l: ]0,1[\rightarrow\mathbb{R}$ and $c\in \mathbb{R}$ such that
\begin{equation}\label{Eq2.3}
\left|f(x)-\left[l(1-x)+c\right]\right|\leq 63 \varepsilon.
\quad \left(x\in ]0,1[\right)
\end{equation}
\end{thm}

\begin{proof}
Define the function $F$ on $]0, +\infty[^{2}$ by
\begin{equation}\label{Eq2.4}
F(u, v)=(u+v)^{\alpha}f\left(\frac{v}{u+v}\right).
\end{equation}
Then
\begin{equation}\label{Eq2.5}
F(tu, tv)=t^{\alpha}F(u, v) \quad \left(t, u, v \in ]0, +\infty[\right)
\end{equation}
and
\begin{equation}\label{Eq2.6}
f(x)=F(1-x, x), \quad \left(x\in ]0,1[\right)
\end{equation}
furthermore, with the substitutions
\[
x=\frac{w}{u+v+w}, \quad y=\frac{v}{u+v+w} \quad \left(u, v, w\in ]0, +\infty[\right)
\]
inequality (\ref{Eq2.1}) implies that
\[
\begin{array}{l}
\left|f\left(\frac{w}{u+v+w}\right)+\frac{(u+v)^{\alpha}}{(u+v+w)^{\alpha}}f\left(\frac{v}{u+v}\right)\right. \\
\left. -f\left(\frac{v}{u+v+w}\right)-\frac{(u+w)^{\alpha}}{(u+v+w)^{\alpha}}f\left(\frac{w}{u+w}\right)\right|
\leq \varepsilon
\end{array}
\]
whence, by (\ref{Eq2.4})
\begin{equation}\label{Eq2.7}
\left|F(u+v, w)+F(u, v)-F(u+w, v)-F(u, w)\right|\leq \varepsilon (u+v+w)^{\alpha}
\end{equation}
follows for all $u, v, w\in ]0, +\infty[$.

In the next step we define the functions $g$ and $G$ on
$]0, +\infty[$ and on $]0, +\infty[^{2}$, respectively by
\begin{equation}\label{Eq2.8}
g(u)=F(u, 1)-F(1, u)
\end{equation}
and
\begin{equation}\label{Eq2.9}
G(u, v)=F(u, v)+g(v).
\end{equation}
We will show that
\begin{equation}\label{Eq2.10}
\left|G(u, v)-G(v, u)\right|\leq 3\varepsilon (u+v+1)^{\alpha}.
\quad \left(u, v\in ]0, +\infty[\right)
\end{equation}
Indeed, with the substitution $w=1$, inequality (\ref{Eq2.7})
implies that
\begin{equation}\label{Eq2.11}
\left|F(u+v, 1)+F(u, v)-F(u+1, v)-F(u, 1)\right|\leq \varepsilon (u+v+1)^{\alpha}.
\end{equation}
Interchanging $u$ and $v$, it follows from (\ref{Eq2.11}) that
\[
\left|-F(u+v, 1)-F(v, u)+F(v+1, u)-F(v, 1)\right|\leq \varepsilon (u+v+1)^{\alpha}.
\quad \left(u, v\in ]0, +\infty[\right)
\]
This inequality, together with (\ref{Eq2.11}) and the triangle inequality
imply that
\begin{equation}\label{Eq2.12}
\left|F(u, v)-F(v, u)-F(u+1, v)-F(u, 1)+F(v+1, u)+F(v, 1)\right|\leq 2\varepsilon (u+v+1)^{\alpha}
\end{equation}
holds for all $u, v\in ]0, +\infty[$.
On the other hand, with $u=1$, we get from (\ref{Eq2.7}) that
\[
\left|F(1+v, w)+F(1, v)-F(1+w, v)-F(1, w)\right|\leq \varepsilon(1+v+w)^{\alpha}.
\]
Replacing here $v$ by $u$ and $w$ by $v$, respectively, we have that
\[
\left|F(u+1, v)+F(1, u)-F(v+1, u)-F(1, v)\right|\leq \varepsilon (u+v+1)^{\alpha}.
\quad \left(u, v\in ]0, +\infty[\right)
\]
Again, the triangle inequality and the definitions
(\ref{Eq2.8}) and (\ref{Eq2.9}), (\ref{Eq2.12}) imply (\ref{Eq2.10}).

In what follows we will investigate the function $g$.
At this point of the proof we have to distinguish two cases.

In case $\alpha<0$ we will determine the function $g$ by proving that
\begin{equation}\label{Eq2.13}
g(u)=c(u^{\alpha}-1) \quad \left(u\in ]0, +\infty[\right)
\end{equation}
with some $c\in \mathbb{R}$.
Indeed, (\ref{Eq2.10}) implies that
\[
\left|G(tu, tv)-G(tv, tu)\right|\leq 3\varepsilon (tu+tv+1)^{\alpha},
\quad \left(t, u, v\in ]0, +\infty[\right)
\]
therefore by (\ref{Eq2.5}) and (\ref{Eq2.9})
\[
\left|t^{\alpha}F(u, v)+g(tv)-t^{\alpha}F(v, u)-g(tu)\right| \leq
3\varepsilon (tu+tv +1)^{\alpha}
\quad \left(t, u, v\in ]0, +\infty[\right)
\]
or
\[
\left|F(u, v)-F(v, u)-t^{-\alpha}\left(g(tu)-g(tv)\right)\right|
\leq 3\varepsilon (u+v+t^{-1})^{\alpha}
\quad \left(t, u, v\in ]0, +\infty[\right)
\]
whence
\[
\lim_{t\rightarrow 0}t^{-\alpha}\left(g(tu)-g(tv)\right)=F(u, v)-F(v, u)
\quad \left(t, u, v\in ]0, +\infty[\right)
\]
follows.
Particularly, with $v=1$, by (\ref{Eq2.8}), we have that
\begin{equation}\label{Eq2.14}
g(u)=\lim_{t\rightarrow 0}t^{-\alpha}\left(g(tu)-g(t)\right).
\quad \left(u\in ]0, +\infty[\right)
\end{equation}
Let now $u, v\in ]0, +\infty[$. Then, by (\ref{Eq2.14}), we obtain that
\[
\begin{array}{rcl}
g(uv)&=& \lim_{t\rightarrow 0}t^{-\alpha}\left[g(tuv)-g(t)\right] \\
 &=&\lim_{t\rightarrow 0}\left[(tv)^{-\alpha}\left(g((tv)u)-g(tv)\right)v^{\alpha}+
 t^{-\alpha}(g(tv)-g(t))\right] \\
  &=&g(u)v^{\alpha}+g(v).
\end{array}
\]
Therefore, $g(u)v^{\alpha}+g(v)=g(v)u^{\alpha}+g(u)$, that is,
\[
g(u)\left(v^{\alpha}-1\right)=g(v)\left(u^{\alpha}-1\right)
\quad \left(u, v\in ]0, +\infty[\right)
\]
which implies (\ref{Eq2.13}) with $c=g(2)\left(2^{\alpha}-1\right)^{-1}$.

Thus, by (\ref{Eq2.6}), (\ref{Eq2.13}), (\ref{Eq2.9}) and (\ref{Eq2.10}), we have that
\begin{equation}\label{Eq2.15}
\begin{array}{l}
\left|f(x)-c(1-x)^{\alpha}-\left(f(1-x)-cx^{\alpha}\right)\right| \\
=\left|F(1-x, x)+cx^{\alpha}-\left(F(x, 1-x)+c(1-x)^{\alpha}\right)\right| \\
=\left|G(1-x, x)-G(x, 1-x)\right|\leq 3\cdot 2^{\alpha}\varepsilon
\end{array}
\end{equation}
holds for all $x\in ]0, 1[$.

In the next step we define the functions $f_{0}$ and
$F_{0}$ on $]0, 1[$ and on $]0, 1[^{2}$ by
\begin{equation}\label{Eq2.16}
f_{0}(x)=f(x)-c\left[(1-x)^{\alpha}-1\right]
\end{equation}
and
\begin{equation}\label{Eq2.17}
F_{0}(p, q)=f_{0}(p)+p^{\alpha}f_{0}(q)-f_{0}(pq)-
(1-pq)^{\alpha}f_{0}\left(\frac{1-p}{1-pq}\right),
\end{equation}
respectively.
Then (\ref{Eq2.1}) and (\ref{Eq2.15}) imply that
\begin{equation}\label{Eq2.18}
\left|f_{0}(x)+(1-x)^{\alpha}f_{0}\left(\frac{y}{1-x}\right)
-f_{0}(y)-(1-y)^{\alpha}f_{0}\left(\frac{x}{1-y}\right)\right|\leq \varepsilon
\end{equation}
for all $(x, y)\in D^{\circ}$ and
\begin{equation}\label{Eq2.19}
\left|f_{0}(x)-f_{0}(1-x)\right|\leq 3\cdot 2^{\alpha}\varepsilon.
\quad \left(x\in ]0, 1[\right)
\end{equation}
Furthermore, with the substitutions $x=1-p$, $y=pq$ ($p, q\in ]0, 1[$), (\ref{Eq2.18})
implies that
\begin{equation}\label{Eq2.20}
\left|f_{0}(1-p)+p^{\alpha}f_{0}(q)-
f_{0}(pq)-(1-pq)^{\alpha}f_{0}\left(\frac{1-p}{1-pq}\right)\right|\leq\varepsilon
\end{equation}
holds for all $p, q\in ]0, 1[$. Therefore, due to (\ref{Eq2.19}) and the
triangle inequality, (\ref{Eq2.18}) implies that
\begin{equation}\label{Eq2.21}
\left|F_{0}(p, q)\right|\leq \left(1+ 3\cdot 2^{\alpha}\right)\varepsilon.
\quad \left(p, q\in ]0, 1[\right)
\end{equation}
It can easily be checked that
\[
\begin{array}{l}
f_{0}(p)\left[q^{\alpha}+(1-q)^{\alpha}-1\right]-f_{0}(q)\left[p^{\alpha}+(1-p)^{\alpha}-1\right] \\
=F_{0}(q, p)-F_{0}(p, q)+(1-pq)^{\alpha}
\left[F_{0}\left(\frac{1-q}{1-pq}, p\right)+
f_{0}\left(1-\frac{1-p}{1-pq}\right)
-f_{0}\left(\frac{1-p}{1-pq}\right)\right]
\end{array}
\]
holds for all $p, q\in ]0, 1[$.
Thus, by (\ref{Eq2.21}) and (\ref{Eq2.19}) we get that
\[
\begin{array}{l}
\left|f_{0}(p)\left[q^{\alpha}+(1-q)^{\alpha}-1\right]-f_{0}(q)\left[p^{\alpha}+(1-p)^{\alpha}-1\right]\right| \\
\leq 2(1+3\cdot 2^{\alpha})\varepsilon
+(1-pq)^{\alpha}\left[(1+3\cdot2^{\alpha})\varepsilon+3\cdot 2^{\alpha}\varepsilon\right],
\end{array}
\]
that is,
\[
\begin{array}{l}
\left|f_{0}(p)-\frac{f_{0}(q)}{q^{\alpha}+(1-q)^{\alpha}-1}
\left[p^{\alpha}+(1-p)^{\alpha}-1\right]\right| \\
\leq
\frac{2(1+3\cdot 2^{\alpha})+(1-pq)^{\alpha}(1+6\cdot 2^{\alpha})}{q^{\alpha}+(1-q)^{\alpha}-1}\varepsilon.
\end{array}
\quad \left(p, q\in ]0, 1[\right)
\]
Taking into consideration (\ref{Eq2.16}), with $q=\frac{1}{2}$ with the definitions
$a=f_{0}\left(\frac{1}{2}\right)\left(2^{1-\alpha}-1\right)^{-1}$, $b=a+c$,
this inequality implies that
\[
\left|f(x)-\left[ax^{\alpha}+b(1-x)^{\alpha}-b\right]\right| \leq
\frac{8+6\cdot 2^{\alpha}+2^{-\alpha}}{2^{1-\alpha}-1}\varepsilon.
\quad  \left(x\in ]0, 1[\right)
\]
Since
\[
\sup_{\alpha <0}\frac{8+6\cdot 2^{\alpha}+2^{-\alpha}}{2^{1-\alpha}-1}=15
\]
we get (\ref{Eq2.2}).

In case $\alpha=0$ we will show that there exists a logarithmic function
$l:]0, +\infty[\rightarrow\mathbb{R}$ such that
\[
\left|g(u)-l(u)\right|\leq 6\varepsilon
\]
for all $u\in ]0, +\infty[$.
Indeed, (\ref{Eq2.10}) yields in this case that
\[
\left|G(u, v)-G(v, u)\right|\leq 3\varepsilon.
\quad \left(u, v\in ]0, +\infty[\right)
\]
Due to (\ref{Eq2.5}) and (\ref{Eq2.9}) we obtain that
\[
\begin{array}{rcl}
G(tu, tv)&=&F(tu, tv)+g(tv)\\
&=&F(u, v)+g(tv)\\
&=&G(u, v)-g(v)+g(tv)
\end{array}
\]
that is,
\[
G(tu, tv)-G(u, v)=g(tv)-g(v),  \quad \left(t, u, v\in ]0, +\infty[\right)
\]
therefore
\begin{equation}\label{Eq2.22}
\begin{array}{l}
\left|g(tv)-g(v)+g(u)-g(tu)\right| \\
=\left|G(tu, tv)-G(u, v)-G(tv, tu)+G(v, u)\right| \\
\leq \left|G(tu, tv)-G(tv, tu)\right|+\left|G(v, u)-G(u, v)\right|
\leq 6\varepsilon
\end{array}
\end{equation}
for all $t, u, v\in ]0, +\infty[$. Now (\ref{Eq2.22}) with the substitution
$u=1$ implies that
\[
\left|g(tv)-g(v)-g(t)\right|\leq 6\varepsilon
\]
holds for all $t, v\in ]0, +\infty[$, since obviously $g(1)=0$.
This means that the function $g$ is approximately logarithmic on
$]0, +\infty[$. Thus (see e.g. Forti \cite{For95}) there exists
a logarithmic function $l:]0, +\infty[\rightarrow\mathbb{R}$ such that
\[
\left|g(u)-l(u)\right|\leq 6\varepsilon
\]
holds for all $u\in ]0, +\infty[$.

Furthermore,
\begin{equation}\label{Eq2.23}
\begin{array}{l}
\left|f(x)-l(1-x)-\left(f(1-x)-l(x)\right)\right| \\
=\left|F(1-x, x)-l(1-x)-F(x, 1-x)+l(x)\right| \\
=\left|F(1-x, x)+g(x)-g(x)-l(1-x)\right. \\
-F(x, 1-x)+g(1-x)-g(1-x)+l(x)\left.\right| \\
\leq \left|F(1-x, x)+g(x)-\left(F(x, 1-x)+g(1-x)\right)\right| \\
+\left|g(1-x)-l(1-x)\right|+\left|l(x)-g(x)\right| \\
=\left|G(1-x, x)-G(x, 1-x)\right|\\
+\left|g(1-x)-l(1-x)\right|+\left|l(x)-g(x)\right| \\
\leq 3\varepsilon+6\varepsilon+6\varepsilon=15 \varepsilon.
\end{array}
\end{equation}
As in the first part of the proof define the functions $f_{0}$ and
$F_{0}$ on $]0, 1[$ and on $]0, 1[^{2}$, respectively, by
\[
f_{0}(x)=f(x)-l(1-x)
\]
and
\[
F_{0}(p, q)=f_{0}(p)+f_{0}(q)-f_{0}(pq)-f\left(\frac{1-p}{1-pq}\right).
\]
Due to (\ref{Eq2.23})
\begin{equation}\label{Eq2.24}
\left|f_{0}(x)-f_{0}(1-x)\right|\leq 15\varepsilon
\end{equation}
holds for all $x\in ]0, 1[$.
Furthermore, inequality (\ref{Eq2.1}) implies,
with the substitutions $x=1-p$, $y=pq$ ($p, q\in ]0, 1[$), that
\begin{equation}\label{Eq2.25}
\left|f_{0}(1-p)+f_{0}(q)-
f_{0}(pq)-f_{0}\left(\frac{1-p}{1-pq}\right)\right|\leq\varepsilon
\end{equation}
holds for all $p, q\in ]0, 1[$.
Inequalities (\ref{Eq2.24})  and (\ref{Eq2.25}) and the triangle inequality
imply that
\begin{equation}\label{Eq2.26}
\left|F_{0}(p, q)\right|\leq 16\varepsilon
\end{equation}
for all $p, q\in ]0, 1[$.
An easy calculation shows that
\[
f_{0}(p)-f_{0}(q)=
F_{0}(q, p)-F_{0}(p, q)+F_{0}\left(\frac{1-p}{1-pq}, p\right)+
f_{0}\left(1-\frac{1-p}{1-pq}\right)-f_{0}\left(\frac{1-p}{1-pq}\right)
\]
therefore,
\begin{equation}\label{Eq2.27}
\begin{array}{l}
\left|f_{0}(p)-f_{0}(q)\right|\\
\leq
\left|F_{0}(q, p)\right|+\left|F_{0}(p, q)\right|+
\left|F_{0}\left(\frac{1-p}{1-pq}, p\right)\right|+
\left|f_{0}\left(1-\frac{1-p}{1-pq}\right)-f_{0}\left(\frac{1-p}{1-pq}\right)\right| \\
\leq 3\cdot 16\varepsilon+15\varepsilon=63\varepsilon
\end{array}
\end{equation}
holds for all $p, q\in ]0, 1[$.
With the substitution $q=\frac{1}{2}$ (\ref{Eq2.27}) implies that
\[
\left|f_{0}(p)-f_{0}\left(\frac{1}{2}\right)\right| \leq 63\varepsilon.
\]
Using the definition of the function $f_{0}$, we obtain that
\[
\left|f(x)-l(1-x)-c\right|\leq 63\varepsilon
\]
holds for all $x\in]0, 1[$, where $c=f_{0}\left(\frac{1}{2}\right)$.
Hence inequality (\ref{Eq2.3}) holds, indeed.
\end{proof}

\begin{rem}
Applying Theorem \ref{Thm2.1} in the case $\varepsilon =0$
we get the general solution of (\ref{Eq1.1}) on $D^{\circ}$
(see also Maksa \cite{Mak82}).
\end{rem}

\section{Two corollaries of the main result}

The first corollary says that equation (\ref{Eq1.1}) is stable on $D$, as well.

\begin{thm}
Let $\alpha, \varepsilon\in\mathbb{R}$ be fixed, $\alpha\leq 0$, $\varepsilon\geq 0$.
Suppose that the function $f:[0,1]\rightarrow\mathbb{R}$ satisfies inequality
(\ref{Eq2.1}) for all $(x, y)\in D$.
Then, in case $\alpha<0$ there
exist $a, b\in\mathbb{R}$ such that the function $h_{1}$
defined on $[0, 1]$ by
\[
h_{1}(x)=\left\{
\begin{array}{lcl}
0, & \text{if} & x=0\\
ax^{\alpha}+b(1-x)^{\alpha}-b, &\text{if}& x\in \left.]0, 1[\right. \\
a-b, & \text{if} & x=1
\end{array}
\right.
\]
is a solution of (\ref{Eq1.1}) on $D$ and
\begin{equation}\label{Eq3.1}
\left|f(x)-h_{1}(x)\right|\leq 15\varepsilon, \quad \left(x\in [0, 1]\right)
\end{equation}
furthermore, in case $\alpha=0$, there exist $a, b, c\in\mathbb{R}$ such
that the function $h_{2}$ defined on $[0, 1]$ by
\[
h_{2}(x)=\left\{
\begin{array}{lcl}
a, & \text{if} & x=0\\
c, &\text{if}& x\in \left.]0, 1[\right. \\
b, & \text{if} & x=1
\end{array}
\right.
\]
is a solution of (\ref{Eq1.1}) on $D$ and
\begin{equation}\label{Eq3.2}
\left|f(x)-h_{2}(x)\right|\leq 63\varepsilon. 
\end{equation}
\end{thm}

\begin{proof}
An easy calculation shows that the functions $h_{1}$ and
$h_{2}$ are the solutions of (\ref{Eq1.1}) on $D$ in
case $\alpha<0$ and in case $\alpha=0$, respectively.
Firstly, we investigate the case $\alpha<0$.
Theorem \ref{Thm2.1}. implies that (\ref{Eq3.1}) holds for
all $x\in ]0, 1[$. Therefore, it is enough to prove that
(\ref{Eq3.1}) holds for $x=0$ and for $x=1$.
It follows from (\ref{Eq2.1}) with $x=0$, that
$\left((1-x)^{\alpha}-1\right)\left|f(0)\right|\leq \varepsilon$ if
$x\in ]0, 1[$. Since $\alpha<0$, $f(0)=0$ follows, that is,
(\ref{Eq3.1}) is valid for $x=0$.
Let now $x\in ]0, 1[$ and $y=1-x$ in (\ref{Eq2.1}).
Then
\begin{equation}\label{Eq3.3}
\left|f(1-x)-f(x)-f(1)\left((1-x)^{\alpha}-x^{\alpha}\right)\right|\leq \varepsilon.
\end{equation}
Apply (\ref{Eq2.2}) to $1-x$ instead of $x$.
Hence we get that
\begin{equation}\label{Eq3.4}
\left|-f(1-x)+a(1-x)^{\alpha}+bx^{\alpha}-b\right|\leq 15\varepsilon
\end{equation}
Adding the inequalities (\ref{Eq3.3}), (\ref{Eq2.2}) and (\ref{Eq3.4}) up
and using the triangle inequality to obtain that
\[
\left|a-b-f(x)\right|\cdot \left|(1-x)^{\alpha}-x^{\alpha}\right|\leq 31\varepsilon.
\quad \left(x\in ]0, 1[\right)
\]
Since $\alpha<0$ we get that $f(1)=a-b$ and so
(\ref{Eq3.1}) holds also for $x=1$.

Now, we fall to deal with the case $\alpha=0$.
Let $x\in ]0, 1[$ and $y=1-x$ in (\ref{Eq2.1}), then we obtain that
\begin{equation}\label{Eq3.5}
\left|f(x)-f(1-x)\right|\leq \varepsilon,  \quad \left(x\in ]0, 1[\right)
\end{equation}
furthermore, let us observe that (\ref{Eq2.1}) does not
pose any restriction on the value of $f(0)$ as well as $f(1)$.
Thus $f(0)=a$, $f(1)=b$, where $a, b\in\mathbb{R}$ are certain constants.
Therefore (\ref{Eq3.2}) holds if $x=0$ or $x=1$.

In case $x\in ]0, 1[$, due to Theorem \ref{Thm2.1}. there
exist a logarithmic function $l:]0, 1[\rightarrow\mathbb{R}$ and
$c\in\mathbb{R}$ such that
\begin{equation}\label{Eq3.6}
\left|f(x)-l(1-x)-c\right|\leq 63\varepsilon
\end{equation}
holds for all $x\in ]0, 1[$. Hence it is enough to prove that the
function $l$ is identically zero on $]0, 1[$.
Indeed, due to (\ref{Eq2.3}), (\ref{Eq3.5}) and (\ref{Eq3.6})
\begin{equation}\label{Eq3.7}
\begin{array}{l}
\left|l(1-x)-l(x)\right| \\
=\left|l(1-x)-f(1-x)+f(1-x)+c\right.\\
-l(x)+f(x)-f(x)-c\left.\right| \\
\leq
\left|l(1-x)+c-f(x)\right|+
\left|f(1-x)-l(x)-c\right|\\
+\left|f(x)-f(1-x)\right| \\
\leq 127 \varepsilon
\end{array}
\end{equation}
holds for all $x\in ]0, 1[$,
using that the function $l$ is logarithmic, the last inequality can be written
as
\[
\left|l\left(\frac{1-x}{x}\right)\right|\leq 127\varepsilon .
\quad \left(x\in ]0, 1[\right)
\]
By substitution $x=\frac{1}{1+p}$ ($p\in ]0, +\infty[$)
into the last inequality, we obtain that
\[
\left|l(p)\right|\leq 127\varepsilon
\]
holds for all $p\in ]0, +\infty[$, where we used the fact that
every logarithmic function on $]0, 1[$ is uniquely extendable to a
logarithmic function on $]0, +\infty[$.

Thus
$l$ is bounded on $]0, +\infty[$.
However, the only bounded, logarithmic function on
$]0, +\infty[$ is the identically zero function.
Therefore,
\[
\left|f(x)-c\right|\leq 63\varepsilon
\]
holds for all $x\in ]0, 1[$, i.e., (\ref{Eq3.2}) is proved.
\end{proof}

The second corollary concerns the stability of a system of
equations.

\begin{thm}\label{Thm3.2}
Let $n\geq 2$ be a fixed positive integer,
\[
\Gamma^{\circ}_{n}=\left\{(p_{1}, \ldots, p_{n})
\vert p_{i}>0, i=1, \ldots, n, \sum^{n}_{i=1}p_{i}=1\right\}
\]
and $(I_{n})$ be the sequence of functions
$I_{n}:\Gamma^{\circ}_{n}\rightarrow\mathbb{R}$
and suppose that there exist a sequence $(\varepsilon_{n})$
of nonnegative real numbers and a real number $\alpha<0$
such that
\begin{equation}\label{Eq3.8}
\left|I_{n}(p_{1}, \ldots, p_{n})-
I_{n-1}(p_{1}+p_{2}, p_{3}, \ldots, p_{n})-
(p_{1}+p_{2})^{\alpha}I_{2}\left(\frac{p_{1}}{p_{1}+p_{2}}, \frac{p_{2}}{p_{1}+p_{2}}\right)\right| \leq
\varepsilon_{n-1}
\end{equation}
for all $n\geq 3$ and $(p_{1}, \ldots, p_{n})\in\Gamma^{\circ}_{n}$, and
\begin{equation}\label{Eq3.9}
\left|I_{3}(p_{1}, p_{2}, p_{3})-I_{3}(p_{1}, p_{3}, p_{2})\right|\leq \varepsilon_{1}
\end{equation}
holds on $\Gamma^{\circ}_{n}$.
Then, in case $\alpha<0$ there exist $c, d\in\mathbb{R}$ such that
\begin{equation}\label{Eq3.10}
\begin{array}{l}
\left|I_{n}\left(p_{1}, \ldots, p_{n}\right)-
\left(c H^{\alpha}_{n}\left(p_{1}, \ldots, p_{n}\right)+d\left(p_{1}^{\alpha}-1\right)\right)\right| \\
\leq \sum^{n-1}_{k=2}\varepsilon_{k}+15\left(2\varepsilon_{2}+\varepsilon_{1}\right)
\left(1+\sum^{n-1}_{k=2}\left(\sum^{k}_{i=1}p_{i}^{\alpha}\right)\right)
\end{array}
\end{equation}
for all $n\geq2$ and $\left(p_{1}, \ldots, p_{n}\right)\in\Gamma^{\circ}_{n}$.
Furthermore, in case $\alpha=0$ there exists a logarithmic function
$l:]0, 1[\rightarrow\mathbb{R}$ and $c\in\mathbb{R}$ such that
\begin{equation}\label{Eq3.11}
\begin{array}{l}
\left|I_{n}\left(p_{1}, \ldots, p_{n}\right)-\left(cH^{0}_{n}\left(p_{1}, \ldots, p_{n}\right)+l(p_{1})\right)\right|\\
\sum^{n-1}_{k=2}\varepsilon_{k}+63\left(n-1\right)\left(2\varepsilon_{2}+\varepsilon_{1}\right)
\end{array}
\end{equation}
for all $n\geq 2$ and $\left(p_{1}, \ldots, p_{n}\right)\in\Gamma^{\circ}_{n}$, where the
convention $\sum^{1}_{k=2}\varepsilon_{k}=\sum^{1}_{k=2}\left(\sum^{k}_{i=1}p_{i}^{\alpha}\right)=0$ is
adapted and
\begin{equation}\label{Eq3.12}
H^{\alpha}_{n}\left(p_{1}, \ldots, p_{n}\right)=
\left(2^{1-\alpha}-1\right)^{-1}\left(\sum^{n}_{i=1}p_{i}^{\alpha}-1\right).
\quad \left(\left(p_{1}, \ldots, p_{n}\right)\in\Gamma^{\circ}_{n}\right)
\end{equation}
\end{thm}

\begin{proof}
As in Maksa \cite{Mak08}, it can be proved that,
due to (\ref{Eq3.8}) and (\ref{Eq3.9}),
for the function $f$ defined on $]0, 1[$ by
$f(x)=I_{2}(1-x, x)$ we get that
\[
\left|f(x)+(1-x)^{\alpha}f\left(\frac{y}{1-x}\right)
-f(y)-(1-y)^{\alpha}f\left(\frac{x}{1-y}\right)\right|\leq 2\varepsilon_{2}+\varepsilon_{1}
\]
for all $(x, y)\in D^{\circ}$, i.e., (\ref{Eq2.1}) holds with
$\varepsilon=2\varepsilon_{2}+\varepsilon_{1}$.
Therefore, applying Theorem \ref{Thm2.1}. we obtain
(\ref{Eq2.2}) and (\ref{Eq2.3}), respectively, with some
$a, b, c\in\mathbb{R}$ and a logarithmic function $l:]0, 1[\rightarrow\mathbb{R}$
and
$\varepsilon=2\varepsilon_{2}+\varepsilon_{1}$, i.e.,

\[
\left|I_{2}\left(1-x, x\right)-\left(ax^{\alpha}+b(1-x)^{\alpha}-b\right)\right|\leq 15(2\varepsilon_{2}+\varepsilon_{1}),
\quad \left(x\in ]0, 1[\right)
\]
in case $\alpha<0$, and
\[
\left|I_{2}\left(1-x, x\right)-\left(l(1-x)+c\right)\right|\leq 63(2\varepsilon_{2}+\varepsilon_{1})
\quad \left(x\in ]0, 1[\right)
\]
in case $\alpha=0$.

Therefore (\ref{Eq3.10}) holds  with $c=(2^{1-\alpha}-1)a$, $d=b-a$ in case $\alpha<0$, and
(\ref{Eq3.11}) holds in case  $\alpha=0$, respectively,  for $n=2$.

We continue the proof by induction on $n$.
Suppose that (\ref{Eq3.10}) and (\ref{Eq3.11}) holds, resp., and for the
sake of brevity, introduce the notation
\[
J_{n}(p_{1}, \ldots, p_{n})=\left\{
\begin{array}{lcl}
cH^{\alpha}_{n}(p_{1}, \ldots, p_{n}),&\text{if}& \alpha<0 \\
cH^{0}_{n}(p_{1}, \ldots, p_{n})+l(p_{1}),&\text{if}& \alpha=0
\end{array}
\right.
\]
for all $n\geq 2$, $(p_{1}, \ldots, p_{n})\in\Gamma^{\circ}_{n}$.
It can easily be seen that (\ref{Eq3.10}) and (\ref{Eq3.11}) hold
on $\Gamma^{\circ}_{n}$ for $J_{n}$ instead of $I_{n}$ ($n\geq 3$)
with $\varepsilon_{n}=0$ ($n\geq 2$).
Thus, for all $(p_{1}, \ldots, p_{n+1})\in\Gamma^{\circ}_{n+1}$, we get that
\[
\begin{array}{l}
I_{n+1}(p_{1}, \ldots, p_{n+1})-J_{n+1}(p_{1}, \ldots, p_{n+1}) \\
=I_{n+1}(p_{1}, \ldots, p_{n+1})-
J_{n}(p_{1}+p_{2}, p_{3}, \ldots, p_{n+1})-
(p_{1}+p_{2})^{\alpha}J_{2}\left(\frac{p_{1}}{p_{1}+p_{2}}, \frac{p_{2}}{p_{1}+p_{2}}\right) \\
=I_{n+1}(p_{1}, \ldots, p_{n+1})-I_{n}(p_{1}+p_{2}, p_{3}, \ldots, p_{n+1})-
(p_{1}+p_{2})^{\alpha}I_{2}\left(\frac{p_{1}}{p_{1}+p_{2}}, \frac{p_{2}}{p_{1}+p_{2}}\right) \\
+I_{n}(p_{1}+p_{2},p_{3}, \ldots, p_{n+1})-J_{n}(p_{1}+p_{2},p_{3}, \ldots, p_{n+1}) \\
+ (p_{1}+p_{2})^{\alpha}\left(I_{2}\left(\frac{p_{1}}{p_{1}+p_{2}}\right)-
J_{2}\left(\frac{p_{1}}{p_{1}+p_{2}}, \frac{p_{2}}{p_{1}+p_{2}}\right)\right).
\end{array}
\]
Therefore, if $\alpha<0$,  (\ref{Eq3.8}) (with $n+1$ instead of $n$),
(\ref{Eq3.10}) with $n=2$ and the induction hypothesis
(applying to $(p_{1}+p_{2}, \ldots, p_{n+1})$ instead of $(p_{1}, \ldots, p_{n})$)
imply that
\[
\begin{array}{l}
\left|I_{n+1}(p_{1}, \ldots, p_{n+1})-J_{n+1}(p_{1}, \ldots, p_{n+1})\right| \\
\leq \varepsilon_{n}+\sum^{n-1}_{k=2}\varepsilon_{k}+
15(2\varepsilon_{2}+\varepsilon_{1})\left(1+\sum^{n-1}_{k=2}\left(\sum^{k+1}_{i=1}p_{i}\right)^{\alpha}\right) \\
+15(2\varepsilon_{2}+\varepsilon_{1})(p_{1}+p_{2})^{\alpha} \\
= \sum^{n}_{k=2}\varepsilon_{k}+15(2\varepsilon_{2}+\varepsilon_{1})
\left(1+\sum^{n}_{k=2}\left(\sum^{k}_{i=1}p_{i}\right)^{\alpha}\right),
\end{array}
\]
that is (\ref{Eq3.10}) holds for $n+1$ instead of $n$.

Finally, if $\alpha=0$, (\ref{Eq3.9}) (with $n+1$ instead of $n$),
(\ref{Eq3.11}) with $n=2$ and the induction hypothesis
(applying to $(p_{1}+p_{2}, \ldots, p_{n+1})$ instead of $(p_{1}, \ldots, p_{n})$)
imply that
\[
\begin{array}{l}
\left|I_{n+1}(p_{1}, \ldots, p_{n+1})-J_{n+1}(p_{1}, \ldots, p_{n+1})\right| \\
\leq \varepsilon_{n}+\sum^{n-1}_{k=2}\varepsilon_{k}+63(n-1)(2\varepsilon_{2}+\varepsilon_{1})+
63(2\varepsilon_{2}+\varepsilon_{1}) \\
=\sum^{n}_{k=2}\varepsilon_{k}+63n(2\varepsilon_{2}+\varepsilon_{1}),
\end{array}
\]
this yields that (\ref{Eq3.11}) holds for $n+1$ instead of $n$.
\end{proof}

\begin{rem}
Applying Theorem \ref{Thm3.2} with the choice
$\varepsilon_{n}=0$ for all $n\in\mathbb{N}$ we get
the $\alpha$--recursive, $3$--semisymmetric information measures.
\end{rem}

\begin{rem}
The sequence $H^{\alpha}_{n}$ defined in (\ref{Eq3.12}) is the entropy
of degree $\alpha$ (see Acz\'{e}l--Dar\'{o}czy \cite{AD75},
Havrda--Charv\'{a}t \cite{HC67}, Dar\'{o}czy \cite{Dar70},
Tsallis \cite{Tsa88}) for $\alpha\leq 0$.
Since the sequence $J_{n}$ satisfies (\ref{Eq3.8}) and (\ref{Eq3.9})
with $\varepsilon_{n}=0$ ($n\geq1$), our theorem says that
$(J_{n})$ is stable in a certain sense.
\end{rem}

\end{document}